\documentclass[11pt]{article}

\usepackage{amsmath,amsfonts,amsthm,fullpage,mathtools,hyperref,xcolor}

\usepackage{palatino,eulervm,eucal}

\usepackage{tikz}

\usetikzlibrary{arrows,calc, scopes, patterns, trees}

\usepgfmodule{shapes}

\usepackage[utf8]{inputenc} %useful to type diacritic characters directly

\numberwithin{equation}{section}
\numberwithin{figure}{section}

\newtheorem{theorem}{Theorem}[section]
\newtheorem{proposition}[theorem]{Proposition}
\newtheorem{lemma}[theorem]{Lemma}
\newtheorem{corollary}[theorem]{Corollary}

{
\theoremstyle{definition}
\newtheorem{definition}[theorem]{Definition}
\newtheorem{example}[theorem]{Example}

}
{
\theoremstyle{remark}

}

\DeclareMathOperator{\pk}{\mathnormal{pk}}

\DeclareMathOperator{\dd}{\mathnormal{dd}}
\DeclareMathOperator{\wpk}{\mathnormal{wpk}}

\DeclareMathOperator{\wdd}{\mathnormal{wdd}}
\DeclareMathOperator{\st}{\mathnormal{st}}

\definecolor{light-gray}{gray}{0.6}
\definecolor{dark-gray}{gray}{0.4}

\newcommand{\us}{-- ++(1,1)}
\newcommand{\vs}[1]{node[fill=white,circle,draw,inner sep=1pt]{$#1$}}
\newcommand{\ds}{-- ++(1,-2)}
\newcommand{\ls}[1]{-- ++(#1,0)}
\newcommand{\mls}{\ls{1}}

\newcommand{\mds}{-- ++(1,-1)}

\makeatletter
\newcommand{\subjclass}[2][1991]{%
  \let\@oldtitle\@title%
  \gdef\@title{\@oldtitle\footnotetext{#1 \emph{Mathematics subject classification:} #2.}}%
}
\newcommand{\keywords}[1]{%
  \let\@@oldtitle\@title%
  \gdef\@title{\@@oldtitle\footnotetext{\emph{Key words and phrases:} #1.}}%
}
\makeatother

\begin{document}

\title{Distribution of peak heights modulo $k$ and double descents on $k$-Dyck paths}

\author{
\textsc{Alexander Burstein}\\
Department of Mathematics\\
Howard University\\
Washington, DC 20059, USA\\
\texttt{aburstein@howard.edu} 
}

\date{February 27, 2023}



\maketitle

\begin{abstract}
We show that the distribution of the number of peaks at height $i$ modulo $k$ in $k$-Dyck paths of a given length is independent of $i\in[0,k-1]$ and is the reversal of the distribution of the total number of peaks. Moreover, these statistics, together with the number of double descents, are jointly equidistributed with any of their permutations. We also generalize this result to generalized Motzkin paths and generalized ballot paths.
\end{abstract}

\section{Introduction} \label{sec:intro}

Given a positive integer $k$, a \emph{$k$-Dyck path} of \emph{down-size} $n$ is a lattice path from $(0,0)$ to $((k+1)n,0)$ with unit steps $u=(1,1)$ (an \emph{up-step}) and $d=(1,-k)$ (a \emph{down-step}) that stays in the first quadrant. This is a generalization of the well-known Dyck paths (see e.g.\ \cite{StEC2}), where $k=1$ and the down-size is called the \emph{semilength}. Note that a $k$-Dyck path of down-size $n$ contains exactly $n$ down-steps and $kn$ up-steps. 

A $ud$ block in a $k$-Dyck path is called a \emph{peak}, and the \emph{height} of that peak is defined as the second coordinate of the left endpoint of the $d$ step in the peak. Given a peak $ud$, we also call the vertex between the $u$ and $d$ steps a peak. Similarly, a $dd$ block (and the vertex in the middle of it) is called a \emph{double descent}. Any Dyck path $P$ with at least one edge has at least one peak, so define $\pk(P)$ to be the number of non-rightmost peaks of $P$ (that is, the number of peaks minus one). Similarly, define $\dd(P)$ to be the number of double descents of $P$.

Deutsch \cite{Deutsch99} showed that the statistics $\pk$ and $\dd$ are equidistributed on the set of Dyck paths of any given length and gave an involution on the set of Dyck paths that interchanged these two statistics. In other words, the bistatistics $(\pk,\dd)$ and $(\dd,\pk)$ are jointly equidistributed on the set of Dyck paths. We generalize this result greatly, first to $k$-Dyck paths, then to the generalizations colored Motzkin and Schr\"{o}der paths. Our results also extend those of DeJager et al.\ \cite{DNSD}, who considered, in particular, $k$-Dyck paths and colored Motzkin and Schr\"{o}der paths where all peaks have a fixed \emph{peak parity} $i$ modulo $k$. Likewise, we generalize the peak enumeration results of Mansour and Shattuck \cite{MS13} for $(k,a)$-Dyck paths.

In Section \ref{sec:main} of this paper we will consider a group of statistics enumerating the heights of peaks of each peak parity $i$ modulo $k$ in $k$-Dyck paths ($0\le i\le k-1$), and prove that they not only have the same distribution, but also have the same joint distribution as any permutation of themselves. Furthermore, these statistics are naturally related to the number-of-$i$-th-children statistics on $(k+1)$-ary trees for $i=1,2,\dots,k+1$, which are easily seen to have the same joint distribution as any of their permutations.

In Section \ref{sec:ext} we will extend and generalize the results of Section \ref{sec:main} to include level steps (e.g.\ generalized colored Motzkin and Schr\"{o}der paths) and some paths that do not end at their initial height (e.g.\ generalized colored ballot paths).

\section{Main results} \label{sec:main}

Let $\mathcal{P}^k_n$ be the set of $k$-Dyck paths of down-size $n\ge 0$. Note that the rightmost peak of any path $P\in\mathcal{P}^k_n$ is followed only by down-steps and thus is at height $0\negthickspace\mod k$. Therefore, we will only consider non-rightmost peaks of paths in $\mathcal{P}^k_n$. 

We define the following statistics on $\mathcal{P}^k_n$. Given a $k$-Dyck path $P\in\mathcal{P}^k_n$, for each $i=0,1,\dots,k-1$, let $\pk(P)$ be the total number of non-rightmost peaks in $P$, and let $\pk_i(P)$ be the number of non-rightmost peaks in $P$ that are at heights $i\negthickspace\mod k$. (Of course, the rightmost peak of $P$ is always at height $0\negthickspace\mod k$, since it is followed only by down-steps, but we will need the values of the above statistics on some subpaths of $P$ as well, where the height modulo $k$ of the rightmost peak may be different.) We also define $\dd(P)$ to be the number of double descents in $P$. Note that every $d$ step of $P$ follows either a $u$ step or a $d$ step, and thus $\dd(P)=n-1-\pk(P)$ for $n\ge 1$ (we subtract $1$ for the rightmost peak of $P$). Then we have the following result.

\begin{theorem} \label{thm:peak-dist}
The $(k+1)$-statistic $(\pk_0,\pk_1,\dots,\pk_{k-1},\dd)$ on $\mathcal{P}^k_n$ is jointly equidistributed with any of its permutations.
\end{theorem}

We note here that because of the total symmetry of these statistics modulo $k$, it may appear useful to consider paths on a cylinder of circumference $k$ rather than on a plane. However, the paths in the plane also stay on or above the $x$-axis, and the corresponding restriction in paths on a cylinder would be more difficult to track. Thus, we stay with the paths in the plane.

We obtain this result by giving a recursive bijection between $k$-Dyck paths and $(k+1)$-ary trees that maps the above statistics on $k$-Dyck paths to the number of $(i+1)$-st children statistics on $(k+1)$-ary trees for $i=0,1,\dots,k-1,k$.

An $l$-ary tree is a rooted ordered tree in which each node has at most $l$ children assigned positions in $\{1,2,\dots,l\}$ from left to right, not necessarily consecutively.

Let $\mathcal{T}^{l}_n$ be the set of $l$-ary trees on $n$ vertices, $n\ge 1$. Define the following statistics on $\mathcal{T}^{l}_n$. Given an $l$-ary tree $T\in\mathcal{T}^{l}_n$, let $e_i(T)$ be the number of $i$-th children in $T$, i.e.\ nodes that are children in position $i$ from the left, for $1\le i\le l$. Equivalently, $e_i(T)$ is the number of $i$-th edges in $T$, i.e.\ edges from a vertex to its child in position $i$.

\begin{proposition} \label{prop:trees}
The $l$-statistic $(e_1,e_2,\dots,e_l)$ on $\mathcal{T}^{l}_n$ is jointly equidistributed with any of its permutations.
\end{proposition}

\begin{proof}
This is almost immediately obvious. It is enough to recursively permute the subtrees of each node of $T$ starting from the root according to the permutation of the statistics $e_i$, $i=1,\dots,l$. We will leave the details of this bijection to the reader.
\end{proof}

We can now connect Proposition \ref{prop:trees} to Theorem \ref{thm:peak-dist}. We will begin by defining two operations on paths.

\begin{definition} \label{def:lifting}
Given a $k$-Dyck path $Q$ starting at height $0$, define its \emph{$i$-lifting} $Q^{(i)}$ as the path with the same sequence of steps but starting at height $i$. 
\end{definition}

Note that any $k$-Dyck path $Q$ with the maximal down-step suffix of length $m\ge 1$ has a unique decomposition as
\begin{equation} \label{eq:decomp-right-peak}
Q=Q_{0}^{(0)}uQ_{1}^{(1)}u\dots Q_{km-1}^{(km-1)}u\underbrace{dd\dots d}_{m},
\end{equation}
where each $Q_{j}$, $0\le j\le km-1$, is a $k$-Dyck path.

\begin{definition} \label{def:cyclic-shift}
Let $Q$ be a $k$-Dyck path with the maximal down-step suffix of length $m\ge 1$ and with the decomposition as in \eqref{eq:decomp-right-peak}. Define a permutation $\pi_m$ on the set $[0,km-1]$ by
\[
\pi_m=\bigl(k-1,\dots,1,0\bigr)\bigl(2k-1,\dots,k+1,k\bigr)\dots\bigl(km-1,\dots,k(m-1)+1,k(m-1)\bigr)
\]
in cycle notation. Define the \emph{cyclic shift} $\kappa(Q)$ of $Q$ by replacing each $Q_j^{(j)}$ with $Q_{\pi_m(j)}^{(j)}$, where $j\in[0,km-1]$. In other words,
\[
Q_j^{(j)}\mapsto 
\begin{cases}
Q_{j+k-1}^{(j)}, & \text{if} \ \ j\equiv 0\!\!\pmod k,\\[4pt]
Q_{j-1}^{(j)}, & \text{otherwise}.  
\end{cases}
\]
\end{definition}

Note that $\kappa$ is a bijection on $\mathcal{P}^k_n$ since $\pi_m$ is a bijection for any $m\ge 1$,
and that $\kappa^{k}$ is the identity map on $\mathcal{P}^k_n$ for any $n$.

\begin{lemma} \label{lem:cyclic-peaks}
For any $k$-Dyck path $Q$ with the maximal down-step suffix of length $m\ge 1$, the map ${\lambda:Q^{(1)}\mapsto\kappa(Q)}$ preserves peak heights modulo $k$ in each segment $Q_j$, $0\le j\le km-1$, as well as the number of double descents in $Q$.
\end{lemma}

We will call $\lambda(Q^{(1)})=\kappa(Q)$ the \emph{cyclic lowering} of $Q^{(1)}$.\footnote{This is not a unique way to define $\lambda$ so as to preserve non-rightmost peak heights modulo $k$. Another possible choice, among many, is to cyclically shift the entire sequence of subpaths $Q_{j}$, i.e.\ to define the permutation $\pi_m$ as the cycle $(km-1,km-2,\dots,1,0)$. However, in that case, $\kappa$ would have order $km$, which depends on $m$. As we defined it, $\pi_m$ has order $k$, independent of $m$.} 

\begin{proof}
Decomposing $Q$ as in \eqref{eq:decomp-right-peak} and lifting it by 1 unit, we see that each segment $Q_j$ in $Q^{(1)}$ starts at height $j+1$. On the other hand, the starting height of the segment $Q_j$ in $\kappa(Q)$ is $j-k+1$ if $j\equiv k-1\!\!\pmod k$, or $j+1$ if $j\not\equiv k-1\!\!\pmod k$. In other words, the height of peaks either does not change or decreases by $k$, and thus remains the same modulo $k$.

Thus, $\lambda$ preserves the statistics $\pk_i$ ($0\le i\le k-1$) on $Q$, and hence their sum $\pk$. But $\pk+\dd=n-1$ on $\mathcal{P}^{k}_{n}$ for any $n\ge 0$, and therefore $\lambda$ preserves the number of double descents as well.
\end{proof}

Applying cyclic lowering several times, we obtain an easy corollary of Lemma \ref{lem:cyclic-peaks}. 
\begin{corollary} \label{cor:cyclic-peaks}
For any $k$-Dyck path $Q$ with the maximal down-step suffix of length $m\ge 1$,  the map $\lambda^i:Q^{(i)}\mapsto\kappa^i(Q)$ on $k$-Dyck paths preserves peak heights modulo $k$ in each block $Q_j$, $0\le j\le km-1$. 
\end{corollary}

We will call $\lambda^i(Q^{(i)})=\kappa^i(Q)$ the \emph{cyclic $i$-lowering} of $Q^{(i)}$. Note that $\kappa^i(Q)$ permutes blocks $Q_j$ as follows:
\[
Q_j^{(j)}\mapsto 
\begin{cases}
Q_{j+k-i}^{(j)}, & \text{if} \quad j\equiv 0,1,\dots,i-1\!\!\!\pmod k,\\[4pt]
Q_{j-i}^{(j)}, & \text{if} \quad j\equiv i,i+1,\dots,k-1\!\!\!\pmod k.  
\end{cases}
\]

\begin{theorem} \label{thm:path-tree}
The statistics $(\pk_0,\pk_1,\dots,\pk_{k-2},\pk_{k-1},\dd)$ on $\mathcal{P}^k_n$ and $(e_1,e_2,\dots,e_{k+1})$ on $\mathcal{T}^{k+1}_n$ are equidistributed.
\end{theorem}

\begin{proof}
Any path $P\in\mathcal{P}^{k}_n$ has a unique decomposition as
\begin{equation} \label{eq:ukd-decomp}
P=P_0^{(0)}uP_1^{(1)}u\dots P_{k-1}^{(k-1)}uP_k^{(k)}d,
\end{equation}
where each $P_i$, $0\le i\le k$, is a $k$-Dyck path. Let $n_i$ be the down-size of $P_i$, then $\sum_{i=0}^{k}{n_i}=n-1$. 

Let $\mathcal{P}^k=\cup_{n\ge 0}{\mathcal{P}^k_n}$ and $\mathcal{T}^{k+1}=\cup_{n\ge 0}{\mathcal{T}^{k+1}_n}$. To prove the theorem, we will give a bijection $\psi:\mathcal{P}^k\to\mathcal{T}^{k+1}$ that maps peaks at heights $i$ modulo $k$ ($i=0,1,\dots,k-1$) of a $k$-Dyck path $P$ of down-size $n$ to $(i+1)$-st children of a $(k+1)$-ary tree $\psi(P)$ on $n$ vertices, and maps double descents of $P$ onto $(k+1)$-st (i.e.\ rightmost) children in $\psi(P)$.

To do this, first label the peaks and double descents of $P$ (i.e.\ the left endpoints of all down-steps of $P$) from left to right as follows:
\begin{itemize}

\item label the rightmost peak by $r$,

\item label the left endpoint of a down-step $i_j$ if it is the $j$-th non-rightmost peak at height $i$ modulo $k$, where $0\le i\le k-1$, and 

\item label the left endpoint of a down-step $d_j$ if it is the $j$-th double descent of $P$.

\end{itemize}

For a tree $T\in\mathcal{T}^{k+1}$, let $R$ be its root label, $T_i$ be its $i$-th subtree, and write $T=R(T_1,T_2,\dots,T_{k+1})$. Then define the map $\psi$ as follows: let $\psi(\emptyset)=\emptyset$ (i.e.\ the one-point path corresponds to the empty tree), and for $P\ne\emptyset$, define $\psi(P)$ recursively so that $R=r$ and, for $i=0,1,\dots,k$, the $(i+1)$-st subtree of $\psi(P)$ is $\psi(P)_{i+1}=\psi(\kappa^i(P_i))$. In other words,
\begin{equation} \label{eq:path-tree}
\psi(P) = r\left(\psi(P_0),\psi(\kappa(P_1)),\psi(\kappa^2(P_2)),\dots,\psi(\kappa^{k-1}(P_{k-1})),\psi(\kappa^k(P_k)\right),
\end{equation}
and we recall that $\kappa^k(P_k)=P_k$. As the root of $\psi(P)$ is already labeled with $r$, label the rest of the vertices of $\psi(P)$ recursively as follows:
\begin{itemize}

\item for $i=0,1,\dots,k-1$, label the $(i+1)$-st child of the root of $\psi(P)$ (i.e.\ the root of $\psi(\kappa^i(P_i))$) with the label of the rightmost peak of $P_i^{(i)}$ (which is at height $i$ modulo $k$).

\item label the $(k+1)$-st (i.e.\ the rightmost) child of the root of $\psi(P)$ (i.e.\ the root of $\psi(P_k)$) with the label of the right endpoint of $P_k^{(k)}$ (which is a double descent if $P_k\ne\emptyset$, since it is preceded and followed by a down-step). Recall that the rightmost peak of $P_k^{(k)}$ is also the rightmost peak of $P$, and hence its label has already been used for the root of $\psi(P)$.

\end{itemize}

Now, proceeding inductively for each $i=0,1,\dots,k$, label the remaining (nonroot) vertices of each $\psi(\kappa^i(P_i))$  with the labels of the corresponding non-rightmost peaks or double descents of $P_i^{(i)}$. Note that the cyclic $i$-lowering $\lambda^i$ may permute positions of peaks and double descents of $P_i^{(i)}$ but preserves peak heights modulo $k$. Therefore, all $(i+1)$-st children in $\psi(P)$ are labeled either $i_j$ (for some $j$) if $0\le i\le k-1$, or $d_j$ (for some $j$) if $i=k$ (see Figure \ref{fig:map-path-tree} and Example \ref{ex:map-path-tree}).

In particular, if $P$ has down-size $n$, then $\psi(P)$ has $n$ vertices. Moreover,
\[
(\pk_0,\pk_1,\dots,\pk_{k-1},\dd)(P)=(e_1,e_2,\dots,e_{k+1})(\psi(P)). \qedhere
\]
\end{proof}

Finally, we see that Proposition \ref{prop:trees} and Theorem \ref{thm:path-tree} together imply Theorem \ref{thm:peak-dist}.

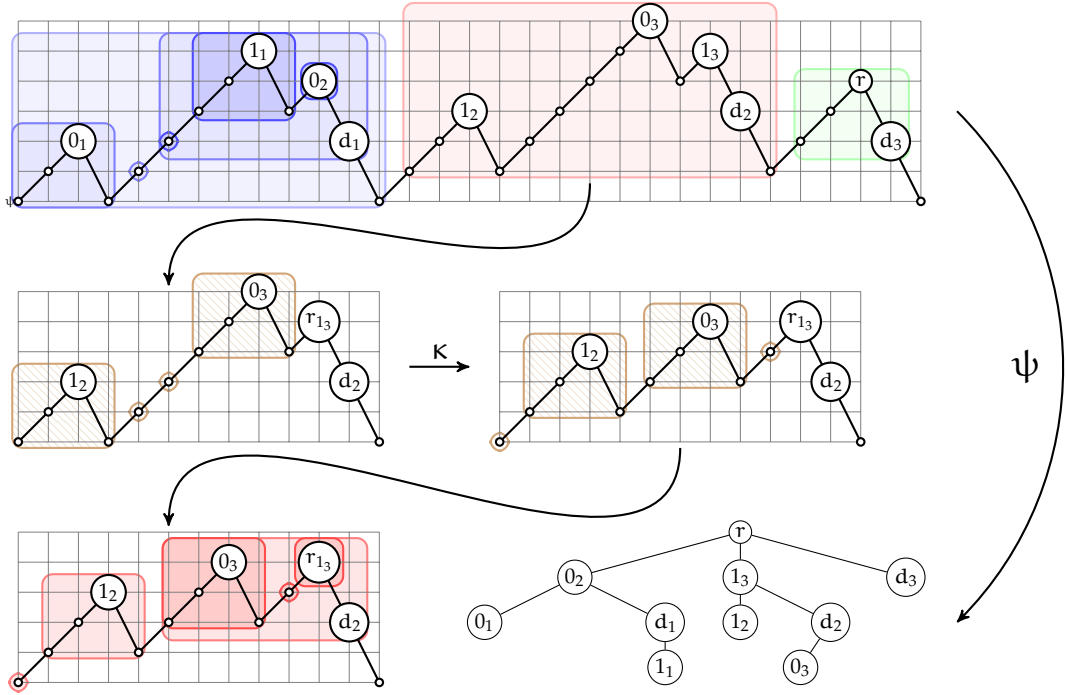
\begin{figure}[t] %\label{fig:map-path-tree}
\begin{center}
{\scriptsize
\begin{tikzpicture}[scale=0.36]
\begin{scope}[xshift=0cm,yshift=0cm,local bounding box=bigpath]
\draw[thick,rounded corners, blue!30,fill=blue!5] (0,0) +(-0.2,-0.2) rectangle +(12.2,5.6);
\draw[thick,rounded corners, red!30,fill=red!5] (13,1) +(-0.2,-0.2) rectangle +(12.2,5.6);
\draw[thick,rounded corners, green!30,fill=green!5] (26,2) +(-0.2,-0.6) rectangle +(3.6,2.4);

\draw[thick,rounded corners, blue!50,fill=blue!10] (0,0) +(-0.2,-0.2) rectangle +(3.2,2.6);
\draw[thick,rounded corners, blue!50,fill=blue!10] (4,1) +(-0.3,-0.3) rectangle +(0.3,0.3);
\draw[thick,rounded corners, blue!50,fill=blue!10] (5,2) +(-0.3,-0.6) rectangle +(6.6,3.6);

\draw[thick,rounded corners, blue!70,fill=blue!20] (5,2) +(-0.3,-0.3) rectangle +(0.3,0.3);
\draw[thick,rounded corners, blue!70,fill=blue!20] (6,3) +(-0.2,-0.3) rectangle +(3.2,2.6);
\draw[thick,rounded corners, blue!70,fill=blue!20] (10,4) +(-0.6,-0.6) rectangle +(0.6,0.6);

\draw[help lines] (0,0) grid (30,6);
\draw[thick] (0,0) 
\vs{} \us \vs{} \us \vs{0_1} \ds  \vs{} \us \vs{} \us \vs{} \us \vs{} \us \vs{} \us \vs{1_1} \ds \vs{} \us \vs{0_2} \ds \vs{d_1} \ds \vs{}
\us
\vs{} \us \vs{} \us \vs{1_2} \ds  \vs{} \us \vs{} \us \vs{} \us \vs{} \us \vs{} \us \vs{0_3} \ds \vs{} \us \vs{1_3} \ds \vs{d_2} \ds \vs{}
\us
\vs{} \us \vs{} \us \vs{r} \ds \vs{d_3}
\ds \vs{}
;
\end{scope}

\draw [thick,out=270,in=90,->,>=stealth'] (19,0.6) to (5,-2.8);

\begin{scope}[xshift=0cm,yshift=-8cm, local bounding box=midpath]

\draw[thick,rounded corners, brown!70,fill=brown!20, pattern=north west lines, pattern color=brown!20] (0,0) +(-0.2,-0.2) rectangle +(3.2,2.6);
\draw[thick,rounded corners, brown!70,fill=brown!20, pattern=north west lines, pattern color=brown!20] (4,1) +(-0.3,-0.3) rectangle +(0.3,0.3);
\draw[thick,rounded corners, brown!70,fill=brown!20, pattern=north west lines, pattern color=brown!20] (5,2) +(-0.3,-0.3) rectangle +(0.3,0.3);
\draw[thick,rounded corners, brown!70,fill=brown!20, pattern=north west lines, pattern color=brown!20] (6,3) +(-0.2,-0.2) rectangle +(3.2,2.6);

\draw[help lines] (0,0) grid (12,5);
\draw[thick] (0,0)
\vs{} \us \vs{} \us \vs{1_2} \ds  \vs{} \us \vs{} \us \vs{} \us \vs{} \us \vs{} \us \vs{0_3} \ds \vs{} \us \vs{r_{1_3}} \ds \vs{d_2} \ds \vs{};

\end{scope}

\draw[thick,->,>=stealth'] (13,-5.5) -- (15,-5.5) node[draw=none,fill=none,font=\small,midway,above] {\large $\kappa$}
;

\begin{scope}[xshift=16cm,yshift=-8cm, local bounding box=midpath1]

\draw[thick,rounded corners, brown!70,fill=brown!20, pattern=north west lines, pattern color=brown!20] (0,0) +(-0.3,-0.3) rectangle +(0.3,0.3);
\draw[thick,rounded corners, brown!70,fill=brown!20, pattern=north west lines, pattern color=brown!20] (1,1) +(-0.2,-0.2) rectangle +(3.2,2.6);
\draw[thick,rounded corners, brown!70,fill=brown!20, pattern=north west lines, pattern color=brown!20] (5,2) +(-0.2,-0.2) rectangle +(3.2,2.6);
\draw[thick,rounded corners, brown!70,fill=brown!20, pattern=north west lines, pattern color=brown!20] (9,3) +(-0.3,-0.3) rectangle +(0.3,0.3);

\draw[help lines] (0,0) grid (12,5);
\draw[thick] (0,0)
\vs{} \us \vs{} \us \vs{} \us \vs{1_2} \ds \vs{} \us \vs{} \us \vs{} \us \vs{0_3} \ds \vs{} \us \vs{} \us \vs{r_{1_3}} \ds \vs{d_2} \ds \vs{};

\end{scope}

\draw [thick,out=270,in=90,->,>=stealth'] (22,-8.2) to (5,-10.8);

\begin{scope}[xshift=0cm,yshift=-16cm, local bounding box=midpath2]

\draw[thick,rounded corners, red!50,fill=red!10] (0,0) +(-0.3,-0.3) rectangle +(0.3,0.3);
\draw[thick,rounded corners, red!50,fill=red!10] (1,1) +(-0.2,-0.2) rectangle +(3.2,2.6);
\draw[thick,rounded corners, red!50,fill=red!10] (5,2) +(-0.2,-0.6) rectangle +(6.6,2.8);

\draw[thick,rounded corners, red!70,fill=red!20] (5,2) +(-0.2,-0.2) rectangle +(3.2,2.8);
\draw[thick,rounded corners, red!70,fill=red!20] (9,3) +(-0.3,-0.3) rectangle +(0.3,0.3);
\draw[thick,rounded corners, red!70,fill=red!20] (10,4) +(-0.8,-0.8) rectangle +(0.8,0.8);

\draw[help lines] (0,0) grid (12,5);
\draw[thick] (0,0)
\vs{} \us \vs{} \us \vs{} \us \vs{1_2} \ds \vs{} \us \vs{} \us \vs{} \us \vs{0_3} \ds \vs{} \us \vs{} \us \vs{r_{1_3}} \ds \vs{d_2} \ds \vs{};

\end{scope}

\begin{scope}[xshift=24cm,yshift=-11cm, 
local bounding box=tree,
every node/.style={circle,draw,inner sep=1pt},
level 1/.style={sibling distance=5.5cm},
level 2/.style={sibling distance=3cm},
level 3/.style={sibling distance=1cm}
]

\node (r) {$r$}
      child { node (02) {$0_2$}
                  child { node (01) {$0_1$} }
                  child { node[draw=none] (none-l-m) {} edge from parent[draw=none] }
                  child { node (d1) {$d_1$} 
                        child { node[draw=none] (none-l-r-l) {} edge from parent[draw=none] }
                        child { node (11) {$1_1$} }
                        child { node[draw=none] (none-l-r-r) {} edge from parent[draw=none] }
                        }
            }
      child { node (13) {$1_3$} 
                  child { edge from parent[draw=none] node[draw=none] (none-m-l) {} }
                  child { node (12) {$1_2$} }
                  child { node (d2) {$d_2$} 
                        child { node (03) {$0_3$} }
                        child { node[draw=none] (none-m-r-m) {} edge from parent[draw=none] }
                        child { node[draw=none] (none-m-r-r) {} edge from parent[draw=none] }
                        }
            }
      child { node (d3) {$d_3$} };
\end{scope}

\draw[thick,->,>=stealth'] (31.2,3) to[bend left=45] (31.2,-14) node[draw=none,fill=none,font=\small,midway,left] {\large $\psi$}
;

\node at (33.5,-5.5) {\Large $\psi$};

\end{tikzpicture}
}
\end{center}
\caption{An example of the application of map $\psi$, together with the corresponding mapping of the peaks and double descents of a $2$-Dyck path $P$ to the nodes of the ternary tree $\psi(P)$.}  
\label{fig:map-path-tree}
\end{figure}

\medskip

\begin{example} \label{ex:map-path-tree}
Figure \ref{fig:map-path-tree} gives an example of the application of map $\psi$ of Theorem \ref{thm:path-tree}, together with the corresponding mapping of the peaks and double descents of a path $P$ to the nodes of the corresponding tree $\psi(P)$. We label the peaks of $P$ as described in Theorem~\ref{thm:path-tree}.  In our example, $k=2$, $n=10$, and
\[
P=\underbrace{uuduuuuududd}_{P_0^{\phantom{(0)}}}u\underbrace{uuduuuuududd}_{P_1^{(1)}}u\underbrace{uud}_{P_2^{(2)}}d.
\]
Then $\kappa^2=\mathrm{id}$ and $\psi(P)=\left(\psi(P_0),\psi(\kappa(P_1)),\psi(\kappa^2(P_2))\right)=\left(\psi(P_0),\psi(\kappa(P_1)),\psi(P_2)\right)$.
For $P_1=uuduuuuududd$, we use the other decomposition,
\[
P_1=\underbrace{uud}_{P_{1,0}^{\phantom{(0)}}}u\underbrace{\emptyset}_{P_{1,1}^{(1)}}u\underbrace{\emptyset}_{P_{1,2}^{(2)}}u\underbrace{uud}_{P_{1,3}^{(3)}}udd.
\]
Since the $ud\dots d$ suffix of $P_1$ contains 2 $d$'s, we have that $\kappa$ permutes the blocks $P_{1,j}$, $j=0,1,2,3$, according to $\pi_2=(10)(32)$, i.e.
\[
\kappa(P_1)=\underbrace{\emptyset}_{P_{1,1}^{(0)}}u\underbrace{uud}_{P_{1,0}^{(1)}}u\underbrace{uud}_{P_{1,3}^{(2)}}u\underbrace{\emptyset}_{P_{1,2}^{(3)}}udd
=uuuduuuduudd.
\]
Note also that $P_0=P_1$ in this example, but the subtrees $\psi(P)_1=\psi(P_0)$ and $\psi(P)_2=\psi(\kappa(P_1))$ of $\psi(P)$ corresponding to $P_0$ and $P_1$, respectively, are not equal as ordered trees.
\end{example}

\medskip

As a direct consequence of Theorem \ref{thm:path-tree}, we have the following corollary.

\begin{corollary} \label{cor:peak-dr-rt}
The statistics $\pk_0$, $\pk_1$, \dots, $\pk_{k-1}$, $\dd$ are equidistributed on $\mathcal{P}^k_n$ for $n\ge 0$.
\end{corollary}

Moreover, the fact that $\dd(P)=n-1-\pk(P)$ for $P\in\mathcal{P}^k_n$, $n\ge 1$, implies that the distribution of each of $\pk_0$, $\pk_1$, \dots, $\pk_{k-1}$, $\dd$ on $\mathcal{P}^k_n$ is the reversed distribution of $\pk$ on $\mathcal{P}^k_n$.

\begin{example} \label{ex:1-dyck}
Letting $k=1$, we obtain that $(\dd,\pk)$ and $(\pk,\dd)$ are equidistributed on Dyck paths. In fact, the well-known bijection due to Deutsch \cite{Deutsch99}, given by $\eta(\emptyset)=\emptyset$, $\eta(P_0uP_1d)\mapsto \eta(P_1)u\eta(P_0)d$ on Dyck paths (more precisely, an involution) interchanges the $\dd$ and $\pk$ statistics. Since the distribution of $\pk$ is given by the Narayana numbers $N(n,r)=\frac{1}{n}\binom{n}{r}\binom{n}{r-1}$, $1\le r\le n$ (see \href{http://oeis.org/A001263}{A001263} \cite{OEIS}), this yields (yet another) bijective proof that $N(n,r)=N(n,n-r)$.
\end{example}

\begin{example} \label{ex:2-dyck}
Letting $k=2$, we see that $\dd$, $\pk_1$ (odd-height non-rightmost peaks), and $\pk_0$ (even-height non-rightmost peaks) are equidistributed on 2-Dyck paths, also known as \emph{ternary paths} (see Figure \ref{fig:ternary}). The distribution of each of $\pk_0$, $\pk_1$, and $\dd$ is given by \href{https://oeis.org/A120986}{A120986} \cite{OEIS}, and the distribution of $\pk=\pk_0+\pk_1=n-1-\dd$ is given by \href{https://oeis.org/A108767}{A108767} \cite{OEIS}. 
\end{example}

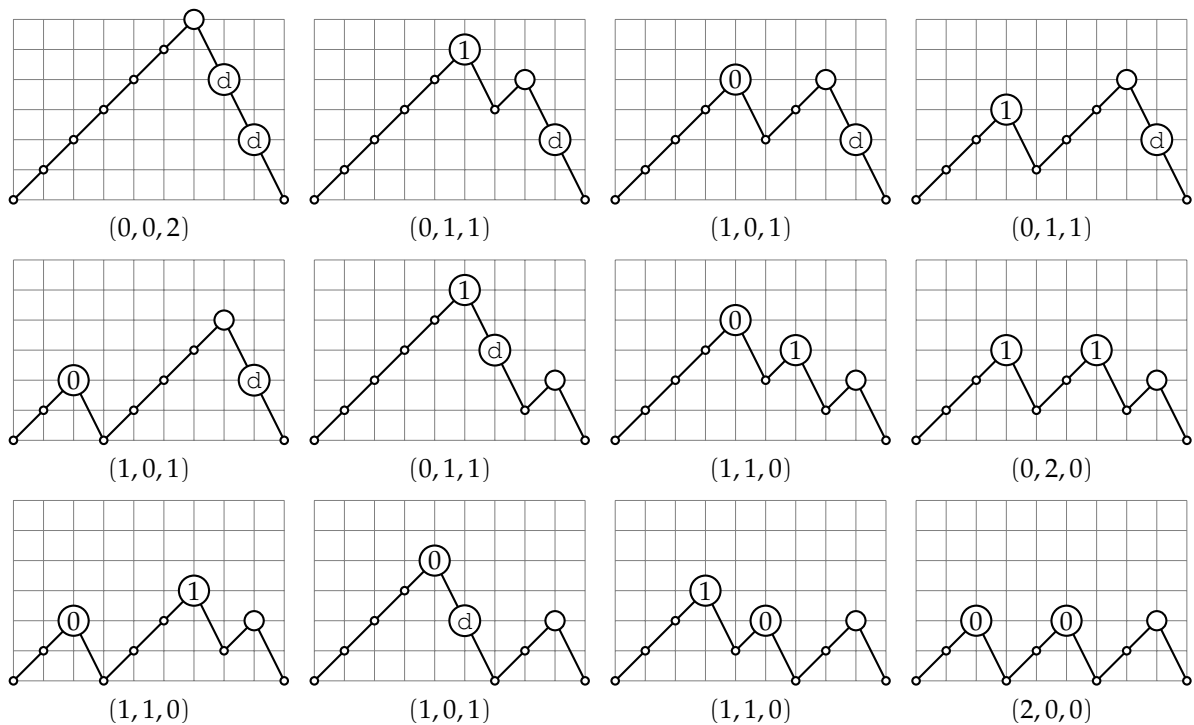
\begin{figure}[t] 
\begin{center}
\begin{tikzpicture}[scale=0.36]
\tikzset{every node}=[font=\small]
\begin{scope}[xshift=0cm,yshift=0cm]
\draw[help lines] (0,0) grid (9,6);
\draw[thick] (0,0) \vs{} \us \vs{} \us \vs{} \us \vs{} \us \vs{} \us \vs{} \us \vs{\ \ } \ds \vs{\mathtt{d}} \ds \vs{\mathtt{d}} \ds \vs{};
\node at (4.5,-1) {$(0,0,2)$};
\end{scope}
\begin{scope}[xshift=10cm,yshift=0cm]
\draw[help lines] (0,0) grid (9,6);
\draw[thick] (0,0) \vs{} \us \vs{} \us \vs{} \us \vs{} \us \vs{} \us \vs{1} \ds \vs{} \us \vs{\ \ } \ds \vs{\mathtt{d}} \ds \vs{};
\node at (4.5,-1) {$(0,1,1)$};
\end{scope}
\begin{scope}[xshift=20cm,yshift=0cm]
\draw[help lines] (0,0) grid (9,6);
\draw[thick] (0,0) \vs{} \us \vs{} \us \vs{} \us \vs{} \us \vs{0} \ds \vs{} \us \vs{} \us \vs{\ \ } \ds \vs{\mathtt{d}} \ds \vs{};
\node at (4.5,-1) {$(1,0,1)$};
\end{scope}
\begin{scope}[xshift=30cm,yshift=0cm]
\draw[help lines] (0,0) grid (9,6);
\draw[thick] (0,0) \vs{} \us \vs{} \us \vs{} \us \vs{1} \ds \vs{} \us \vs{} \us \vs{} \us \vs{\ \ } \ds \vs{\mathtt{d}} \ds \vs{};
\node at (4.5,-1) {$(0,1,1)$};
\end{scope}
\begin{scope}[xshift=0cm,yshift=-8cm]
\draw[help lines] (0,0) grid (9,6);
\draw[thick] (0,0) \vs{} \us \vs{} \us \vs{0} \ds \vs{} \us \vs{} \us \vs{} \us \vs{} \us \vs{\ \ } \ds \vs{\mathtt{d}} \ds \vs{};
\node at (4.5,-1) {$(1,0,1)$};
\end{scope}
\begin{scope}[xshift=10cm,yshift=-8cm]
\draw[help lines] (0,0) grid (9,6);
\draw[thick] (0,0) \vs{} \us \vs{} \us \vs{} \us \vs{} \us \vs{} \us \vs{1} \ds \vs{\mathtt{d}} \ds \vs{} \us \vs{\ \ } \ds \vs{};
\node at (4.5,-1) {$(0,1,1)$};
\end{scope}
\begin{scope}[xshift=20cm,yshift=-8cm]
\draw[help lines] (0,0) grid (9,6);
\draw[thick] (0,0) \vs{} \us \vs{} \us \vs{} \us \vs{} \us \vs{0} \ds \vs{} \us \vs{1} \ds \vs{} \us \vs{\ \ } \ds \vs{};
\node at (4.5,-1) {$(1,1,0)$};
\end{scope}
\begin{scope}[xshift=30cm,yshift=-8cm]
\draw[help lines] (0,0) grid (9,6);
\draw[thick] (0,0) \vs{} \us \vs{} \us \vs{} \us \vs{1} \ds \vs{} \us \vs{} \us \vs{1} \ds \vs{} \us \vs{\ \ } \ds \vs{};
\node at (4.5,-1) {$(0,2,0)$};
\end{scope}
\begin{scope}[xshift=0cm,yshift=-16cm]
\draw[help lines] (0,0) grid (9,6);
\draw[thick] (0,0) \vs{} \us \vs{} \us \vs{0} \ds \vs{} \us \vs{} \us \vs{} \us \vs{1} \ds \vs{} \us \vs{\ \ } \ds \vs{};
\node at (4.5,-1) {$(1,1,0)$};
\end{scope}
\begin{scope}[xshift=10cm,yshift=-16cm]
\draw[help lines] (0,0) grid (9,6);
\draw[thick] (0,0) \vs{} \us \vs{} \us \vs{} \us \vs{} \us \vs{0} \ds \vs{\mathtt{d}} \ds \vs{} \us \vs{} \us \vs{\ \ } \ds \vs{};
\node at (4.5,-1) {$(1,0,1)$};
\end{scope}
\begin{scope}[xshift=20cm,yshift=-16cm]
\draw[help lines] (0,0) grid (9,6);
\draw[thick] (0,0) \vs{} \us \vs{} \us \vs{} \us \vs{1} \ds \vs{} \us \vs{0} \ds \vs{} \us \vs{} \us \vs{\ \ } \ds \vs{};
\node at (4.5,-1) {$(1,1,0)$};
\end{scope}
\begin{scope}[xshift=30cm,yshift=-16cm]
\draw[help lines] (0,0) grid (9,6);
\draw[thick] (0,0) \vs{} \us \vs{} \us \vs{0} \ds \vs{} \us \vs{} \us \vs{0} \ds \vs{} \us \vs{} \us \vs{\ \ } \ds \vs{};
\node at (4.5,-1) {$(2,0,0)$};
\end{scope}
\end{tikzpicture}
\caption{Peak heights modulo 2 and double descents in 2-Dyck paths of down-size 3. A non-rightmost peak at height $i\!\negthickspace\mod2$ ($i\in\{0,1\}$) is labeled with $i$. A double descent is labeled with $\mathtt{d}$. The rightmost peak is unlabeled. Below each path $P$ is the value of $(\pk_0,\pk_1,\dd)(P)$.}
\label{fig:ternary}
\end{center}
\end{figure}

The preceding results enable us now to calculate the joint distribution on $\mathcal{P}^k_n$ ($n\ge 1$) of the statistic $(\pk_0,\pk_1,\dots,\pk_{k-2},\pk_{k-1},\dd)$.

\begin{corollary} \label{cor:peak-dist}
The number of paths $P\in\mathcal{P}^k_n$ ($n\ge 1$) such that $(\pk_0,\pk_1,\dots,\pk_{k-1},\dd)(P)=(r_0,r_1,\dots,r_{k-1},r_k)$, where $\sum_{i=0}^{k}{r_i}=n-1$, is
\[
\frac{1}{n}\prod_{i=0}^{k}{\binom{n}{r_i}}.
\]
Moreover, for each statistic $\st\in\{\pk_0,\pk_1,\dots,\pk_{k-1},\dd\}$, the number of paths $P\in\mathcal{P}^k_n$ ($n\ge 1$) such that $\st(P)=r$ is
\[
\frac{1}{n}\binom{n}{r}\binom{kn}{n-1-r},
\]
whereas the number of paths $P\in\mathcal{P}^k_n$ ($n\ge 1$) such that $\pk(P)=r+1$ is
\[
\frac{1}{n}\binom{n}{n-1-r}\binom{kn}{r}=\frac{1}{n}\binom{n}{r+1}\binom{kn}{r},
\]
where $0\le r\le n-1$.
\end{corollary}

Note that $\sum_{i=0}^{k}{r_i}=n-1$ in Corollary \ref{cor:peak-dist} since a down-step must be preceded by an up-step or another down-step, so that $\sum_{i=0}^{k}{r_i}$ counts all the down-steps of $P$ except the one immediately after its rightmost peak.

\begin{proof}
Let $f=f(x;q_0,q_1,\dots,q_{k-1},q_k)$ be the generating function over the nonempty $k$-Dyck paths defined by
\begin{equation} \label{eq:path-func}
\begin{split}
f=f(x;q_0,q_1,\dots,q_{k-1},q_k)&=\sum_{n=1}^{\infty}\sum_{P\in\mathcal{P}^k_n}{x^n\left(\prod_{i=0}^{k-1}{q_i^{\pk_i(P)}}\right)q_k^{\dd(P)}}\\
&=\sum_{n=1}^{\infty}\sum_{T\in\mathcal{T}^{k+1}_n}{x^n\left(\prod_{i=0}^{k}{q_i^{e_{i+1}(T)}}\right)}.
\end{split}
\end{equation}
Then Theorem \ref{thm:path-tree} and its proof imply that
\begin{equation} \label{eq:path-func-rec}
f=x\prod_{i=0}^{k}(q_if+1).
\end{equation}
This functional equation follows from the decomposition \eqref{eq:ukd-decomp}. The $k$ up-steps and $1$ down-step not in any $P_i^{(i)}$, $i=0,1,\dots,k$, contribute the factor $x$, while $P_i^{(i)}$ (unless it is a single point) contributes the factor of $q_i$ due to its rightmost peak, which is at height $i\negthickspace\mod k$.

The number of paths $P\in\mathcal{P}^k_n$ such that $(\pk_0,\pk_1,\dots,\pk_{k-1},\dd)(P)=(r_0,r_1,\dots,r_{k-1},r_k)$ is the coefficient of $f$ at $x^n\prod_{i=0}^{k}{q_i^{r_i}}$. Applying Lagrange inversion (see e.g.\ \cite[Theorem 5.1.1]{Wilf}) to \eqref{eq:path-func-rec}, we get that
\[
\left[x^n\prod_{i=0}^{k}{q_i^{r_i}}\right]f=\frac{1}{n}\left[f^{n-1}\prod_{i=0}^{k}{q_i^{r_i}}\right]\left(\prod_{i=0}^{k}(q_if+1)\right)^n=\frac{1}{n}\prod_{i=0}^{k}{\binom{n}{r_i}},
\]
which is the number of paths we want to enumerate.

The second and third statement of the corollary can be derived similarly by setting all but one $q_i=1$ and noting that $\pk=n-1-\dd$ on $\mathcal{P}^k_n$.
\end{proof}

\section{Extensions and generalizations} \label{sec:ext}

\subsection{Paths with level steps} \label{subsec:level}

Theorem \ref{thm:peak-dist} can be extended to generalized colored Motzkin and Schr\"oder paths as follows. Let $A$ be a subset of positive integers, and let $c=(c_a)_{a\ge 1}$ be a sequence of nonnegative integers such that $c_a=0$ exactly when $a\notin A$. Consider a path $P$ from $(0,0)$ back to height 0 with unit steps $u=(1,1)$, $d=(1,-k)$, and if $A\ne\emptyset$, then $l_{a,b}=(a,0)$ of color $b$, where $1\le b\le c_a$, for all $a\in A$. Note that the case $A=\emptyset$ means that we exclude all level steps so as to generate the $k$-Dyck paths considered in Section~\ref{sec:main}. We will call such paths \emph{$(k,A,c)$-paths}, as this generalizes $(k,a)$-paths defined in \cite{MS13} for $A=\{a\}$ and $c_a=1$, and let $\mathcal{P}^{k,A,c}_n$ be the set of $(k,A,c)$-paths of length $n$. We also define the length of up, down, and level steps by $|u|=1$, $|d|=1$, and $|l_{a,b}|=a$, and the length $|P|$ of a path $P$ as the sum of the lengths of its steps.

On such a $(k,A,c)$-path $P$, call the blocks $dd$ and $l_{a,b}d$ \emph{weak double descents}, and call blocks $ud$ and $ul_{a,b}$, as well as the single step $l_{a,b}$ if it is the leftmost step of $P$, \emph{weak peaks} (in each case here, $a\in A$ and $b\in[c_a]$). Extending our peak height definition for $k$-Dyck paths, let the height of a weak peak be the second coordinate of the left endpoint of the down or level step in the peak. Furthermore, on $(k,A,c)$-paths $P$, define the following statistics:
\begin{itemize}

\item $\wpk_i(P)$, the number of non-rightmost weak peaks of $P$ at heights $i\negmedspace\mod k$, where $i=0,1,\dots,k-1$,

\item $\wpk(P)=\sum_{i=0}^{k-1}{\wpk_i(P)}$, the total number of non-rightmost peaks of $P$,

\item $\wdd(P)$, the number of weak double descents of $P$.

\end{itemize}

Then we have the following unique decomposition:
\[
P=L \quad \text{or} \quad P=P_0^{(0)}uP_1^{(1)}u\dots P_{k-1}^{(k-1)}uP_k^{(k)}dL,
\]
where each $P_i$ is a $(k,A,c)$-path of down-size $n_i$ (such that $1+\sum_{i=0}^{k}{n_i}=n$, the down-size of $P$) and $L$ is a block of level steps of allowed lengths and colors. Furthermore, as before, for each path $P_i$ we have the unique decomposition
\[
P_i=P_{i,0}^{(0)}uP_{i,1}^{(1)}u\dots P_{i,km_i-1}^{(km_i-1)}uD_i,
\]
where each $P_{i,j}$ is a $(k,A,c)$-path and $D_i$ is block of down-steps and level steps (i.e.\ no up-steps) containing exactly $m_i$ down-steps and possibly some level steps of any allowed lengths and colors. Then it is easy to see that, for each $i\in[0,k-1]$, the cyclic $i$-lowering map $\lambda_i$ that permutes the subpaths $P_{i,j}$ in the same way as in Corollary \ref{cor:cyclic-peaks}, preserves weak peak heights modulo $k$ as well as the number of weak peaks and weak double descents in $P_i^{(i)}$. The preceding discussion implies the following result.

\begin{theorem} \label{thm:weak-peak-dist}
The $(k+1)$-statistic $(\wpk_0,\wpk_1,\dots,\wpk_{k-1},\wdd)$ on $\mathcal{P}^{k,A,c}_n$ is jointly equidistributed with any of its permutations.
\end{theorem}

It is a rather straightforward exercise to adapt the bijective proof of Theorem~\ref{thm:peak-dist} to prove Theorem~\ref{thm:weak-peak-dist}, and we will leave it to the reader. Instead, we will give a short generating function argument similar to that given in Corollary \ref{cor:peak-dist} so as not to repeat essentially the same proof with minor variations. 

\begin{proof}
Let $c_A(x)=\sum_{a\in A}{c_ax^a}$ and $\mathcal{P}^{k,A,c}=\cup_{n=0}^{\infty}{\mathcal{P}^{k,A,c}_n}$, and
\begin{equation} \label{eq:fkac}
f_{k,A,c}=f_{k,A,c}(x;q_0,q_1,\dots,q_{k-1},q_k)=\sum_{P\in\mathcal{P}^{k,A,c}\setminus\{\emptyset\}}{x^{|P|}\left(\prod_{i=0}^{k-1}{q_i^{\wpk_i(P)}}\right)q_k^{\wdd(P)}}\, .
\end{equation}
Then, from the path decompositions above and the cyclic $i$-lowering maps, we obtain the functional equation
\begin{equation} \label{eq:weak-path-func-rec}
f_{k,A,c}=c_A(x)(f_{k,A,c}+1)+x^{k+1}\prod_{i=0}^{k}(q_if_{k,A,c}+1),
\end{equation}
Now notice that Equation \eqref{eq:weak-path-func-rec} is symmetric in the variables $q_0,q_1\dots,q_k$.
\end{proof}

Note also that $c_A(x)=0$ if $A=\emptyset$, so letting $A=\emptyset$ yields Theorem~\ref{thm:peak-dist}. 

\begin{example} \label{ex:motzkin-schroeder}
Letting $k=1$, we obtain the joint equidistribution of $(\wdd,\wpk)$ and $(\wpk,\wdd)$ on Motzkin paths by setting $A=\{1\}$, $c_1=1$ (see Figure \ref{fig:motzkin} for an example), and on Schr\"{o}der paths by setting $A=\{2\}$, $c_2=1$.
\end{example}

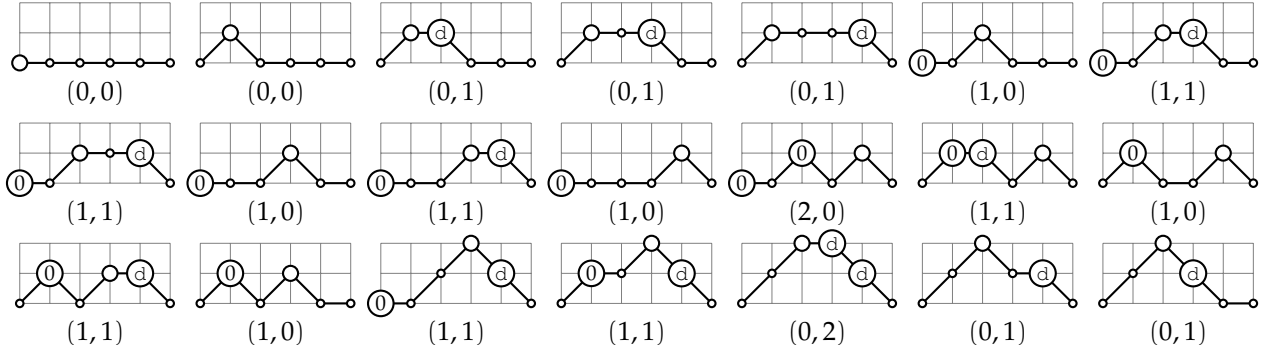
\begin{figure}[t] 
\begin{center}
\begin{tikzpicture}[scale=0.36]
\tikzset{every node}=[font=\scriptsize]
\begin{scope}[xshift=0cm,yshift=0cm]
\draw[help lines] (0,0) grid (5,2);
\draw[thick] (0,0) \vs{\ \,} \mls \vs{} \mls \vs{} \mls \vs{} \mls \vs{} \mls \vs{};
\node at (2.5,-1) {\small $(0,0)$};
\end{scope}
\begin{scope}[xshift=6cm,yshift=0cm]
\draw[help lines] (0,0) grid (5,2);
\draw[thick] (0,0) \vs{} \us \vs{\ \,} \mds \vs{} \mls \vs{} \mls \vs{} \mls \vs{};
\node at (2.5,-1) {\small $(0,0)$};
\end{scope}
\begin{scope}[xshift=12cm,yshift=0cm]
\draw[help lines] (0,0) grid (5,2);
\draw[thick] (0,0) \vs{} \us \vs{\ \,} \mls \vs{\mathtt{d}} \mds \vs{}  \mls \vs{} \mls \vs{};
\node at (2.5,-1) {\small $(0,1)$};
\end{scope}
\begin{scope}[xshift=18cm,yshift=0cm]
\draw[help lines] (0,0) grid (5,2);
\draw[thick] (0,0) \vs{} \us \vs{\ \,} \mls \vs{} \mls \vs{\mathtt{d}} \mds \vs{}  \mls \vs{};
\node at (2.5,-1) {\small $(0,1)$};
\end{scope}
\begin{scope}[xshift=24cm,yshift=0cm]
\draw[help lines] (0,0) grid (5,2);
\draw[thick] (0,0) \vs{} \us \vs{\ \,} \mls \vs{} \mls \vs{} \mls \vs{\mathtt{d}} \mds \vs{};
\node at (2.5,-1) {\small $(0,1)$};
\end{scope}
\begin{scope}[xshift=30cm,yshift=0cm]
\draw[help lines] (0,0) grid (5,2);
\draw[thick] (0,0) \vs{0} \mls \vs{} \us \vs{\ \,} \mds \vs{} \mls \vs{} \mls \vs{};
\node at (2.5,-1) {\small $(1,0)$};
\end{scope}
\begin{scope}[xshift=36cm,yshift=0cm]
\draw[help lines] (0,0) grid (5,2);
\draw[thick] (0,0) \vs{0} \mls \vs{} \us \vs{\ \,} \mls \vs{\mathtt{d}} \mds \vs{} \mls \vs{};
\node at (2.5,-1) {\small $(1,1)$};
\end{scope}
\begin{scope}[xshift=0cm,yshift=-4cm]
\draw[help lines] (0,0) grid (5,2);
\draw[thick] (0,0) \vs{0} \mls \vs{} \us \vs{\ \,} \mls \vs{} \mls \vs{\mathtt{d}} \mds \vs{};
\node at (2.5,-1) {\small $(1,1)$};
\end{scope}
\begin{scope}[xshift=6cm,yshift=-4cm]
\draw[help lines] (0,0) grid (5,2);
\draw[thick] (0,0) \vs{0} \mls \vs{} \mls \vs{} \us \vs{\ \,} \mds \vs{} \mls \vs{};
\node at (2.5,-1) {\small $(1,0)$};
\end{scope}
\begin{scope}[xshift=12cm,yshift=-4cm]
\draw[help lines] (0,0) grid (5,2);
\draw[thick] (0,0) \vs{0} \mls \vs{} \mls \vs{} \us \vs{\ \,} \mls \vs{\mathtt{d}} \mds \vs{};
\node at (2.5,-1) {\small $(1,1)$};
\end{scope}
\begin{scope}[xshift=18cm,yshift=-4cm]
\draw[help lines] (0,0) grid (5,2);
\draw[thick] (0,0) \vs{0} \mls \vs{} \mls \vs{} \mls \vs{} \us \vs{\ \,} \mds \vs{};
\node at (2.5,-1) {\small $(1,0)$};
\end{scope}
\begin{scope}[xshift=24cm,yshift=-4cm]
\draw[help lines] (0,0) grid (5,2);
\draw[thick] (0,0) \vs{0} \mls \vs{} \us \vs{0} \mds \vs{} \us \vs{\ \,} \mds \vs{};
\node at (2.5,-1) {\small $(2,0)$};
\end{scope}
\begin{scope}[xshift=30cm,yshift=-4cm]
\draw[help lines] (0,0) grid (5,2);
\draw[thick] (0,0) \vs{} \us \vs{0} \mls \vs{\mathtt{d}} \mds \vs{} \us \vs{\ \,} \mds \vs{};
\node at (2.5,-1) {\small $(1,1)$};
\end{scope}
\begin{scope}[xshift=36cm,yshift=-4cm]
\draw[help lines] (0,0) grid (5,2);
\draw[thick] (0,0) \vs{} \us \vs{0} \mds \vs{} \mls \vs{} \us \vs{\ \,} \mds \vs{};
\node at (2.5,-1) {\small $(1,0)$};
\end{scope}
\begin{scope}[xshift=0cm,yshift=-8cm]
\draw[help lines] (0,0) grid (5,2);
\draw[thick] (0,0) \vs{} \us \vs{0} \mds \vs{} \us \vs{\ \,} \mls \vs{\mathtt{d}} \mds \vs{};
\node at (2.5,-1) {\small $(1,1)$};
\end{scope}
\begin{scope}[xshift=6cm,yshift=-8cm]
\draw[help lines] (0,0) grid (5,2);
\draw[thick] (0,0) \vs{} \us \vs{0} \mds \vs{} \us \vs{\ \,} \mds \vs{} \mls \vs{};
\node at (2.5,-1) {\small $(1,0)$};
\end{scope}
\begin{scope}[xshift=12cm,yshift=-8cm]
\draw[help lines] (0,0) grid (5,2);
\draw[thick] (0,0) \vs{0} \mls \vs{} \us \vs{} \us \vs{\ \,} \mds \vs{\mathtt{d}} \mds \vs{};
\node at (2.5,-1) {\small $(1,1)$};
\end{scope}
\begin{scope}[xshift=18cm,yshift=-8cm]
\draw[help lines] (0,0) grid (5,2);
\draw[thick] (0,0) \vs{} \us \vs{0} \mls \vs{} \us \vs{\ \,} \mds \vs{\mathtt{d}} \mds \vs{};
\node at (2.5,-1) {\small $(1,1)$};
\end{scope}
\begin{scope}[xshift=24cm,yshift=-8cm]
\draw[help lines] (0,0) grid (5,2);
\draw[thick] (0,0) \vs{} \us \vs{} \us \vs{\ \,} \mls \vs{\mathtt{d}} \mds \vs{\mathtt{d}} \mds \vs{};
\node at (2.5,-1) {\small $(0,2)$};
\end{scope}
\begin{scope}[xshift=30cm,yshift=-8cm]
\draw[help lines] (0,0) grid (5,2);
\draw[thick] (0,0) \vs{} \us \vs{} \us \vs{\ \,} \mds \vs{} \mls \vs{\mathtt{d}} \mds \vs{};
\node at (2.5,-1) {\small $(0,1)$};
\end{scope}
\begin{scope}[xshift=36cm,yshift=-8cm]
\draw[help lines] (0,0) grid (5,2);
\draw[thick] (0,0) \vs{} \us \vs{} \us \vs{\ \,} \mds \vs{\mathtt{d}} \mds  \vs{} \mls\vs{};
\node at (2.5,-1) {\small $(0,1)$};
\end{scope}
\end{tikzpicture}
\caption{Weak peaks and weak double descents in Motzkin paths of length 5. A non-rightmost weak peak is labeled with $0$ (the only remainder modulo $k=1$). A weak double descent is labeled with $\mathtt{d}$. The rightmost weak peak is unlabeled. Below each path $P$ is the value of $(\wpk,\wdd)(P)$.}
\label{fig:motzkin}
\end{center}
\end{figure}

\subsection{Ballot paths} \label{subsec:ballot}

We can also generalize Theorem \ref{thm:peak-dist} to left prefixes of $k$-Dyck paths starting at $(0,0)$ and ending at height less than $k$. Define a \emph{$(k,m)$-ballot path} of down-size $n$ to be a path from $(0,0)$ to $((k+1)n+m,m)$ with steps $u=(1,1)$ and $d=(1,-k)$. Let $\mathcal{P}^{k,m}$ be the set of $(k,m)$-ballot paths, and let $\mathcal{P}^{k,m}_n$ be the set of $(k,m)$-ballot paths of down-size $n$.

Let $m\equiv r\!\pmod k$, so that $m=\ell k +r$ for some integers $\ell\ge 0$ and $0\le r\le k-1$. Our key observation here is that, for $m\not\equiv 0\!\pmod k$, a $(k,m)$-ballot path is a prefix of some $k$-Dyck path, and moreover, it can only be extended on the right to a $k$-Dyck path by using at least some up-steps (as well as down-steps). Thus, any suffix extending a $(k,m)$-ballot path to a $k$-Dyck path must contain a peak. Therefore, we need to modify the statistics we consider by letting $\pk^*_i(P)$ be the number of peaks in $(k,m)$-ballot path $P$ at heights $i\negthickspace\mod k$ (so these statistics count the rightmost peak as well).

\begin{theorem} \label{thm:ballot-peaks}
Let $m,\ell\ge 0$ and $0\le r\le k-1$ be integers such that $m=\ell k +r$. Then for any $0\le r\le k-1$, 
\begin{itemize}
\item the $(r+1)$-statistic $(\pk^*_0,\dots,\pk^*_r)$ on $\mathcal{P}^{k,m}_n$ is jointly equidistributed with any of its permutations.

\item the $(k-r-1)$-statistic $(\pk^*_{r+1},\dots,\pk^*_{k-1})$ on $\mathcal{P}^{k,m}_n$ is jointly equidistributed with any of its permutations.
\end{itemize}
\end{theorem}

\begin{proof}
Note that any $(k,m)$-ballot path $P$ can be uniquely decomposed as
\begin{equation} \label{eq:ballot-decomp}
P=P_0^{(0)}uP_1^{(1)}u\dots uP_{m-1}^{(m-1)}uP_m^{(m)},
\end{equation}
where each $P_i$, $i=0,1,\dots,m$, is a $k$-Dyck path, on which the multistatistics $(\pk_0,\dots,\pk_r)$ and $(\pk_{r+1},\dots,\pk_{k-1})$ are jointly equidistributed with any of their respective permutations. The only peaks not counted by the unstarred statistics are the rightmost peaks. Note that the rightmost peak of each subpath $P_i^{(i)}$ is at height $i\!\!\mod k$. Thus, $\ell+1$ of these subpaths, if nonempty, have their rightmost peak at heights in $[0,r]$ modulo $k$, whereas only $\ell$ of these subpaths, if nonempty, have their rightmost peak at heights in $[r+1,k-1]$ modulo $k$. Thus, for $i=0,1,\dots,k-1$,
\begin{equation} \label{eq:peak-ballot-rec}
\begin{split}
\pk^*_i(P) 
&=\sum_{j=0}^{m}{\pk^*_i(P_j^{(j)})}
=
\begin{cases}
\sum_{j=0}^{m}{\pk_i(P_j^{(j)})}+\sum_{j=0}^{\ell}[P_{kj+i}\ne\emptyset], & \text{ if } 0\le i\le r,\\
\sum_{j=0}^{m}{\pk_i(P_j^{(j)})}+\sum_{j=0}^{\ell-1}[P_{kj+i}\ne\emptyset], & \text{ if } r+1\le i\le k-1,
\end{cases}
\\
&=
\begin{cases}
\sum_{j=0}^{m}{\pk_i(\kappa^j(P_j))}+\sum_{j=0}^{\ell}[P_{kj+i}\ne\emptyset], & \text{ if } 0\le i\le r,\\
\sum_{j=0}^{m}{\pk_i(\kappa^j(P_j))}+\sum_{j=0}^{\ell-1}[P_{kj+i}\ne\emptyset], & \text{ if } r+1\le i\le k-1,
\end{cases}
\end{split}
\end{equation}
where $\kappa$ is the cyclic shift map defined on page \pageref{def:cyclic-shift} and, as before, $[\cdot]$ is the Iverson bracket. 

The distribution of these statistics can be computed as well. Let
\begin{equation} \label{eq:gkm}
g_{k,m}=g_{k,m}(x;q_0,q_1,\dots,q_k)=\sum_{n=0}^{\infty}\sum_{P\in\mathcal{P}^{k,m}_n}{x^n\left(\prod_{i=0}^{k-1}{q_i^{\pk^*_i(P)}}\right)q_k^{\dd(P)}},
\end{equation}
then, from \eqref{eq:peak-ballot-rec},
\begin{equation} \label{eq:gkm-rec}
g_{k,m}=g_{k,\ell k +r}=\left(\prod_{i=0}^{r}{(q_i f+1)}\right)^{\ell+1}\left(\prod_{i=r+1}^{k-1}{(q_i f+1)}\right)^{\ell},
\end{equation}
where the function $f$ is defined in \eqref{eq:path-func} and satisfies \eqref{eq:path-func-rec}. The multipliers $q_i$ in \eqref{eq:gkm-rec} correspond to the fact that the rightmost peak of a nonempty $P_{kj+i}^{(kj+i)}$ is at height $i$ modulo $k$. From \eqref{eq:path-func-rec}, $f$ is symmetric in $q_0,\dots,q_k$, and therefore it is clear from \eqref{eq:gkm-rec} that $g_{k,\ell k+r}$ is symmetric in $q_0,\dots,q_r$ as well as in $q_{r+1},\dots,q_{k-1}$.
\end{proof}

Applying Lagrange inversion to \eqref{eq:gkm-rec} as in Corollary \ref{cor:peak-dist} yields after some routine manipulations that the number of paths $P\in\mathcal{P}^{k,\ell k + r}_n$ with $(\pk^*_0,\dots,\pk^*_{k-1},\dd)(P)=(s_0,\dots,s_{k-1},s_k)$, where $\sum_{i=0}^{k}{s_i}=n\ge 1$, is
\begin{multline} \label{eq:peak-ballot-dist}
\left[x^n\prod_{i=0}^{k}{q_i^{s_i}}\right]g_{k,\ell k + r}
\\=
\frac{1}{n}\left(\frac{\ell+1}{n+\ell+1}\sum_{i=0}^{r}{s_i}+\frac{\ell}{n+\ell}\sum_{i=r+1}^{k-1}{s_i}\right)\left(\prod_{i=0}^{r}{\binom{n+\ell+1}{s_i}}\right)\left(\prod_{i=r+1}^{k-1}{\binom{n+\ell}{s_i}}\right)\binom{n}{s_k}.
\end{multline}

Note that letting $\ell=0$ and $r=0$ in \eqref{eq:peak-ballot-dist} yields the result of Corollary \ref{cor:peak-dist}, since $(r_0,r_1,\dots,r_k)$ in Corollary \ref{cor:peak-dist} is equal to $(s_0-1,s_1,\dots,s_k)$ in the notation of \eqref{eq:peak-ballot-dist}.

\begin{example} \label{ex:ternary-1}
Letting $k=2$ and $r=1$, we see that statistics $(pk^*_0,pk^*_1)$ and $(pk^*_1,pk^*_0)$ are equidistributed on $(2,1)$-ballot paths of down-size $n$. In particular, statistics $\pk^*_0$ (even-height peaks) and $\pk^*_1$ (odd-height peaks) are equidistributed on those paths. For an example of this, see the last 7 paths in Figure \ref{fig:ternary} (rightmost three in row 2 and all of row 3), delete the rightmost $ud$ from each of those, and ignore the $\dd$ statistic.
\end{example}

The above result can be generalized for the weak versions of the corresponding statistics.

Let $A$ and $c$ be defined as for $(k,A,c)$-paths. Define a \emph{$(k,A,c,m)$-path} of down-size $n$ to be a path from height $0$ to height $m\in[0,k-1]$ with steps $u=(1,1)$, $d=(1,-k)$, and any level steps $l_{a,b}=(a,0)$ of color $b$, for $a\in A$ and $b\in[c_a]$. Note that a $(k,A,c,m)$-path of down-size $n$ (i.e.\ with $n$ down-steps) has $kn+m$ up-steps. Let $\mathcal{P}^{k,A,c,r}_n$ be the set of such paths, and let $\mathcal{P}^{k,A,c,m}=\cup_{n=0}^{\infty}{\mathcal{P}^{k,A,c,m}_n}$. Note that $\mathcal{P}^{k,A,c}_n=\mathcal{P}^{k,A,c,0}_n$, and therefore, $\mathcal{P}^{k,A,c}=\mathcal{P}^{k,A,c,0}$. Let $\wpk^*_i(P)$ ($i=0,1,\dots,k-1$) be the number of (possibly rightmost) weak peaks of $P$ at heights $i\negmedspace\mod k$. Then we have the following result.

\begin{theorem} \label{thm:weak-ballot-peaks}
Let $m,\ell\ge 0$ and $0\le r\le k-1$ be integers such that $m=\ell k +r$. Then for any $0\le r\le k-1$, 
\begin{itemize}
\item the $(r+1)$-statistic $(\wpk^*_0,\dots,\wpk^*_r)$ on $\mathcal{P}^{k,A,c,m}_n$ is jointly equidistributed with any of its permutations.

\item the $(k-r-1)$-statistic $(\wpk^*_{r+1},\dots,\wpk^*_{k-1})$ on $\mathcal{P}^{k,A,c,m}_n$ is jointly equidistributed with any of its permutations.
\end{itemize}
\end{theorem}

\begin{proof}
Note that any $(k,A,c,m)$-path $P$ can be uniquely decomposed as
\begin{equation} \label{eq:ballot-decomp-m}
P=P_0^{(0)}uP_1^{(1)}u\dots P_{m-1}^{(m-1)}uP_m^{(m)},
\end{equation}
where each $P_i$, $i=0,1,\dots,m$, is a $(k,A,c)$-path. 
Define
\begin{equation} \label{eq:gkacr}
g_{k,A,c,m}=g_{k,A,c,m}(x;q_0,q_1,\dots,q_k)=\sum_{P\in\mathcal{P}^{k,A,c,r}}{x^{|P|}\left(\prod_{i=0}^{k-1}{q_i^{\wpk^*_i(P)}}\right)q_k^{\wdd(P)}},
\end{equation}
then the decomposition \eqref{eq:ballot-decomp-m} implies, as in the proof of Theorem \ref{thm:ballot-peaks}, that
\[
g_{k,A,c,m}=g_{k,A,c,\ell k+r}=
x^m\left(\prod_{i=0}^{r}{(q_i f_{k,A,c}+1)}\right)^{\ell+1}\left(\prod_{i=r+1}^{k-1}{(q_i f_{k,A,c}+1)}\right)^{\ell},
\]
where $f_{k,A,c}$ is defined in \eqref{eq:fkac} and satisfies \eqref{eq:weak-path-func-rec}, which is symmetric in $q_0,\dots,q_k$. Therefore, $g_{k,A,c,\ell k+r}$ is symmetric in $q_0,\dots,q_r$ and in $q_{r+1},\dots,q_{k-1}$.
\end{proof}

\bigskip

It would also be interesting to find a natural generalization of these results to more general families of lattice paths, including walks and bridges \cite{BF02}, where the paths are allowed go below the $x$-axis. However, given up-steps $(1,1)$, down-steps $(1,-k)$, and an arbitrary set of level steps as unit path steps, the peak height statistics modulo $k$ considered in this paper are no longer equidistributed either on walks (paths that may go below the $x$-axis and end anywhere) or on bridges (walks that end on the $x$-axis).

{\small
\begin{center}
\textsc{Acknowledgements}
\end{center}

The author would like to thank Lou Shapiro, Sergey Kirgizov, Sam Hopkins, and the anonymous referees for their valuable suggestions on improving the presentation of the results.
}

\end{document}